\documentclass[oneside, a4paper, 11pt]{article}
\usepackage[hidelinks]{hyperref}
\usepackage{amsmath, amsthm, amssymb, amsfonts}
\usepackage[left=2.5cm,right=2.5cm,top=3cm,bottom=3cm,a4paper]{geometry}
\usepackage{thmtools}
\newtheorem{thm}{Theorem}[section]

\newtheorem{lem}[thm]{Lemma}

\declaretheorem[numbered=no]{Question}
\usepackage{array}
\makeatletter
\newcommand{\thickhline}{%
    \noalign {\ifnum 0=`}\fi \hrule height 1pt
    \futurelet \reserved@a \@xhline
}
\newcolumntype{"}{@{\hskip\tabcolsep\vrule width 1pt\hskip\tabcolsep}}
\makeatother

\title{\textbf{On a finite subsemigroup of semigroups in which $x^r=x$}}
\author{\large{Jungin Lee}}
\date{}

\usepackage{fancyhdr}

\newcommand\shorttitle{On a finite subsemigroup of semigroups in which $x^r=x$}
\newcommand\authors{Jungin Lee}

\begin{document}
\maketitle

\vspace{0mm}

\textbf{Abstract.} In this paper, we consider the following problem: For every positive integer $r \geq 2$, find all positive integers $n$ such that for every semigroup of order $\geq n$ in which $x^r=x$ for every element $x$ has a subsemigroup of order $n$. 

\vspace{5mm}


\vspace{5mm}

\section{Introduction}

\noindent Green and Rees [\ref{ref:1}] proved that the order of a free idempotent semigroup on $n$ generators is given by $I_n=\sum_{r=1}^{n}\binom{n}{r}\prod_{i=1}^{r}(r-i+1)^{2^i}$. Thus for every idempotent semigroup $S$ of order $\geq n$, the order $x$ of the subsemigroup of $S$ generated by $n$ distinct elements of $S$ satisfies $n \leq x \leq I_n$. There exist only finite number of idempotent semigroups of order $\leq I_n$. This implies that for every positive integer $n$, we can determine that whether every idempotent semigroup $S$ of order $\geq n$ has a subsemigroup of order $n$ by finite number of calculations. Based on this idea we will solve the following problem. 

\begin{Question} \label{quen}
For every positive integer $r \geq 2$, find all positive integers $n$ such that for every semigroup of order $\geq n$ in which $x^r=x$ for every element $x$ has a subsemigroup of order $n$. 
\end{Question}

\vspace{3mm}

\section{Case I : $r=2$}

Every element of an idempotent semigroup forms a subsemigroup of order 1, so the Question is true for $n=1$. In this section, we will show that $n=1,2,4,6$ are the only positive integers which every idempotent semigroup of order $\geq n$ has a subsemigroup of order $n$. 

\vspace{3mm}

\begin{thm} \label{thm2.1}
Every idempotent semigroup $S$ of order $\geq 2$ has a subsemigroup of order $2$.
\end{thm}

\begin{proof}
Suppose that $S$ does not have a subsemigroup of order 2, and let $a$ and $b$ be different elements of $S$. Then $S_1=\left \{ ab,aba \right \}, \,S_2=\left \{ a,ab,aba \right \}$ are subsemigroups of $S$. $\left | S_1 \right |\neq 2$ implies $ab=aba$, and $\left | S_2 \right | =\left | \left \{ a,ab \right \} \right |  \neq 2$ implies $a=ab$. For the same reason, we get $b=ba=bab$. Thus $\left | \left \langle a,b \right \rangle \right |=2$, a contradiction.
\end{proof}

\vspace{3mm}

\begin{lem} \label{lem2.2}
If an idempotent semigroup $S$ of order $\geq 4$ does not have a subsemigroup of order $4$, then $\forall a,b \in S \:\: \left | \left \langle a,b \right \rangle \right | \leq 3$ and $aba,bab \in \left \{ a,b,ab,ba \right \}$. 
\end{lem}

\begin{proof}
It is easy to see that $S_1=\left \{ ab,ba,aba,bab \right \},\, S_2=\left \{ a \right \}\cup S_1$ are subsemigroups of $S$. Suppose that $S$ does not have a subsemigroup of order 4. Then $\left | S_1 \right |\neq 4$ implies $\left | S_1 \right |\leq 3$, and $\left | S_2 \right |\neq  4$ implies $\left | S_2 \right |\leq 3$. Now $\left | \left \langle a,b \right \rangle \right |=\left | S_2\cup \left \{ b \right \} \right |\leq 4$ and $\left | \left \langle a,b \right \rangle \right |\neq 4$, so $\left | \left \langle a,b \right \rangle \right |\leq 3$. If $aba \notin \left \{ a,b,ab,ba \right \}$, $\left \langle a,b \right \rangle=\left \{ a,b,aba \right \}$ and $ab,ba \neq aba$ so $ab,ba \in \left \{ a,b \right \}\: \Rightarrow \: aba=(ab)(ba)\in \left \{ a,b,ab,ba \right \}$, a contradiction. Thus $aba \in \left \{ a,b,ab,ba \right \}$ and for the same reason, $bab \in \left \{ a,b,ab,ba \right \}$. 
\end{proof}

\vspace{3mm}

\begin{thm} \label{thm2.3}
Every idempotent semigroup $S$ of order $\geq 4$ has a subsemigroup of order $4$.
\end{thm}

\begin{proof}
Suppose that $S$ does not have a subsemigroup of order 4. By Theorem \ref{thm2.1}, $S$ has a subsemigroup $\left \{ a,b \right \}$ of order 2. Suppose that $\left | \left \langle a,b,c \right \rangle \right |=3$ for every $c \in S-\left \{ a,b \right \}$. Let $c$ and $d$ be two different elements of $S-\left \{ a,b \right \}$. Then $4\leq \left | \left \{ a,b \right \}\cup \left \langle c,d \right \rangle\right | \leq \left | \left \{ a,b \right \} \right |+\left | \left \langle c,d \right \rangle \right | \leq 5$ by Lemma \ref{lem2.2} and $\left \{ a,b \right \}\cup \left \langle c,d \right \rangle$ is a subsemigroup of $S$, so by the assumption $\left | \left \{ a,b \right \}\cup \left \langle c,d \right \rangle \right |=5$ and $\left \{ a,b \right \}\cap \left \langle c,d \right \rangle=\phi$. Again by theorem \ref{thm2.1}, there exists a subsemigroup $P$ of $\left \langle c,d \right \rangle$ such that $\left | P \right |=2$. Then $\left \{ a,b \right \}\cup P$ is a subsemigroup of $\left \langle c,d \right \rangle$ of order 4, a contradiction. Thus there exists $c \in S$ such that $\left | \left \langle a,b,c \right \rangle \right |\geq 4$. By Lemma \ref{lem2.2}, we have the following 5 cases. \\
\\
(i) $\left \langle a,c \right \rangle=\left \{ a,c \right \}$ (or, $\left \langle b,c \right \rangle=\left \{ b,c \right \}$) : $\left \langle a,b,c \right \rangle=\left \{ a,b,c,bc,cb \right \}=\left \{ a \right \}\cup \left \langle b,c \right \rangle$, so $4\leq \left | \left \langle a,b,c \right \rangle \right |\leq\left | \left \{ a \right \} \right |+\left | \left \langle b,c \right \rangle \right |\leq  4$ by Lemma \ref{lem2.2}, a contradiction. \\
(ii) $\left \langle a,c \right \rangle=\left \{ a,c,ac \right \},\, \left \langle b,c \right \rangle=\left \{ b,c,bc \right \}$ : $\left \langle a,b,c \right \rangle=\left \{ a,b,c,ac,bc \right \}=\left \{ c \right \}\cup \left \{ a,b,ac,bc \right \}$. If $ax=c$ or $bx=c$ for some $x \in S$, then $ac=c$ or $bc=c$, a contradiction. Thus $S_1=\left \{ a,b,ac,bc \right \}$ is a subsemigroup of $\left \langle a,b,c \right \rangle$. $\left | S_1 \right |\neq 4$ implies $\left | S_1 \right |\leq 3$ and $\left | \left \langle a,b,c \right \rangle \right |\leq 4$, a contradiction. \\
(iii) $\left \langle a,c \right \rangle=\left \{ a,c,ac \right \},\, \left \langle b,c \right \rangle=\left \{ b,c,cb \right \}$ : $\left \langle a,b,c \right \rangle=\left \{ a,b,c,ac,cb,acb \right \}$ and it is easy to see that $S_2=\left \{ c,ac,cb,acb \right \},\, S_3=\left \{ b \right \}\cup S_2$ are subsemigroups of $\left \langle a,b,c \right \rangle$. $\left | S_2 \right |\neq 4$ implies $\left | S_2 \right |\leq 3$, and $\left | S_3 \right |\neq4$ implies $\left | S_3 \right |\leq 3$. Thus $\left | \left \langle a,b,c \right \rangle \right |\leq 4$, a contradiction. \\
(iv) $\left \langle a,c \right \rangle=\left \{ a,c,ca \right \},\, \left \langle b,c \right \rangle=\left \{ b,c,bc \right \}$ : symmetric to (iii). \\
(v) $\left \langle a,c \right \rangle=\left \{ a,c,ca \right \},\, \left \langle b,c \right \rangle=\left \{ b,c,cb \right \}$ : symmetric to (ii). 
\end{proof}

\vspace{3mm}

Now we will prove that every idempotent semigroup of order $\geq 6$ has a subsemigroup of order 6. In order to prove this, we will first establish some lemmas. 

\vspace{3mm}

\begin{lem} \label{lem2.4a}
Suppose that an idempotent semigroup $S$ of order $\geq 6$ does not have a subsemigroup of order $6$. If $M$ is a subsemigroup of $S$ of order $4$, then there exists $x \in S-M$ such that $\left | \left \langle M\cup \left \{ x \right \} \right \rangle \right |\geq 6$. 
\end{lem}

\begin{proof}
Suppose that $\left | \left \langle M\cup \left \{ x \right \} \right \rangle \right |=5$ for every $x \in S-M$. Let $a$ and $b$ be two distinct elements of $S-M$. Then $S_1=M\cup \left \{ ab,aba \right \}$ and $S_2=S_1\cup \left \{ a \right \}$ are subsemigroups of $S$. $\left | S_1 \right |\neq 6$ implies $\left | S_1 \right |\leq 5$, and $\left | S_2 \right |\neq 6$ implies $\left | S_2 \right |\leq 5$. Thus $ab,aba \in M\cup \left \{ a \right \}$ and for the same reason, $ba,bab \in M\cup \left \{ b \right \}$. Then $M\cup \left \langle a,b \right \rangle=M\cup \left \{ a,b \right \}$ is a subsemigroup of $S$ of order 6, a contradiction.
\end{proof}

\vspace{3mm}

\begin{lem} \label{lem2.4}
If an idempotent semigroup $S$ of order $\geq 6$ does not have a subsemigroup of order $6$ and $\forall a,b,c \in S\:\: \left | \left \langle a,b,c \right \rangle \right |\leq 5$, then $\forall a,b \in S \:\: \left | \left \langle a,b \right \rangle \right |\leq 3$. Also, there exist $a,b,c \in S$ such that $\left | \left \langle a,b \right \rangle \right |=3$ and $\left | \left \langle a,b,c \right \rangle \right |=4$. 
\end{lem}

\begin{proof}
If there exist $a,b \in S$ such that $\left | \left \langle a,b \right \rangle \right |\geq 5$, then it is trivial that there exist $a,b,c \in S$ such that $\left | \left \langle a,b,c \right \rangle \right |\geq 6$. Suppose that $\left | \left \langle a,b \right \rangle \right |=4$. Then by Lemma \ref{lem2.4a}, there exists $c \in S-\left \langle a,b \right \rangle$ such that $\left | \left \langle a,b,c \right \rangle \right |\geq 6$, a contradiction. \\
If $\left | \left \langle x,y \right \rangle \right |=2$ for every $x,y \in S$, then every subset of $S$ of order 6 is a subsemigroup of $S$, a contradiction. So there exist $x,y \in S$ such that $\left | \left \langle x,y \right \rangle \right |=3$. Suppose that for every $z \in S-\left \langle x,y \right \rangle$, $\left | \left \langle x,y,z \right \rangle \right |=5$. Without loss of generality, let $\left \langle x,y,z \right \rangle$ be $\left \{ x,y,xy,z,xz \right \}$ or $\left \{ x,y,xy,z,zx \right \}$. \\
\\
(i) $\left \langle x,y,z \right \rangle  =\left \{ x,y,xy,z,xz \right \}$ : From the table \ref{tab:1}, if $zxy \neq y$, then $\left \{ x,xy,z,xz \right \}$ is a subsemigroup of $S$. Thus $\left | \left \langle x,z \right \rangle \right |=3$ and $\left | \left \langle x,z,xy \right \rangle \right |=4$. For the same reason, $yzx \neq z$ implies $\left | \left \langle x,y \right \rangle \right |=3$ and $\left | \left \langle x,y,xz \right \rangle \right |=4$. If $zxy=y$ and $yzx=z$, then $\left \{ y,xy,z,xz \right \}$ is a subsemigroup of $S$ and $z(xy)=y$. Thus $\left | \left \langle z,xy \right \rangle \right |=3$ and $\left | \left \langle z,xy,xz \right \rangle \right |=4$. \\
\begin{table}
\small
\centering
\setlength\tabcolsep{2.5pt}
\begin{tabular}{c"c|c|c|c|c}
    $*$ & $x$ & $y$ & $xy$ & $z$ & $xz$ \\\thickhline
    $x$ & $x$ & $xy$ & $xy$ & $xz$ & $xz$ \\\hline
    $y$ & $x,y,xy$ & $y$ & $y,xy$ & $y,xy,z,xz$ & $yzx$ \\\hline
    $xy$ & $x,xy$ & $xy$ & $xy$ & $xy,xz$ & $xy,xz$ \\\hline
    $z$ & $x,z,xz$ & $y,xy,z,xz$ & $zxy$ & $z$ & $z,xz$ \\\hline
    $xz$ & $x,xz$ & $xy,xz$ & $xy,xz$ & $xz$ & $xz$
\end{tabular}
\caption{$\left \langle x,y,z \right \rangle  =\left \{ x,y,xy,z,xz \right \}$}
\label{tab:1} 
\end{table}
\\
(ii) $\left \langle x,y,z \right \rangle  =\left \{ x,y,xy,z,zx \right \}$ : From the table \ref{tab:2}, $zxy \neq y$ implies $\left | \left \langle x,z \right \rangle \right |=3$ and $\left | \left \langle x,z,zx \right \rangle \right |=4$. Also $zxy \neq z$ implies $\left | \left \langle x,y \right \rangle \right |=3$ and $\left | \left \langle x,y,zx \right \rangle \right |=4$.
\begin{table}
\small
\centering
\setlength\tabcolsep{2.5pt}
\begin{tabular}{c"c|c|c|c|c}
    $*$ & $x$ & $y$ & $xy$ & $z$ & $zx$ \\\thickhline
    $x$ & $x$ & $xy$ & $xy$ & $x,z,xz$ & $x,zx$ \\\hline
    $y$ & $x,y,xy$ & $y$ & $y,xy$ & $y,xy,z,zx$ & $x,y,xy,zx$ \\\hline
    $xy$ & $x,xy$ & $xy$ & $xy$ & $x,xy,z,zx$ & $x,xy,zx$ \\\hline
    $z$ & $zx$ & $y,xy,z,zx$ & $zxy$ & $z$ & $zx$ \\\hline
    $zx$ & $zx$ & $zxy$ & $zxy$ & $z,zx$ & $zx$
\end{tabular}
\caption{$\left \langle x,y,z \right \rangle  =\left \{ x,y,xy,z,zx \right \}$}
\label{tab:2} 
\end{table}
\end{proof}

\vspace{3mm}

\begin{lem} \label{lem2.5}
If $a$, $b$ and $c$ are distinct elements of an idempotent semigroup $S$ and $\left | \left \langle a,b,c \right \rangle \right |\leq 5$, then  $abc \in \left \{ a,b,c,ab,ac,bc \right \}$.
\end{lem}

\begin{proof}
Suppose that $abc \notin \left \{ a,b,c,ab,ac,bc \right \}$. Then $\left | \left \{ a,b,c,ab,ac,bc \right \} \right |\leq 4$, so either $ab \in \left \{ a,b,c \right \}$, $bc \in \left \{ a,b,c \right \}$ or $ab=bc$. Each of them implies $abc \in \left \{ ac,bc,c \right \}$, $abc \in \left \{ a,ab,ac \right \}$ and $abc=bc$, respectively. 
\end{proof}

\vspace{3mm}

\begin{lem} \label{lem2.5a}
Let $a$ and $b$ be two different elements of an idempotent semigroup $S$ such that $\left | \left \langle a,b \right \rangle \right |=4$. Then $\left \langle a,b \right \rangle =\left \{ a,b,ab,ba \right \}$ and exactly one of the following holds. \\
(1) $aba=a$, $bab=b$ \\
(2) $aba=ab$, $bab=ba$ \\
(3) $aba=ba$, $bab=ab$
\end{lem}

\begin{proof}
If $\left | \left \{ a,b,ab,ba \right \} \right | \leq3$, then $aba,bab \in \left \{ a,b,ab,ba \right \}$ and $\left | \left \langle a,b \right \rangle \right | \leq3$, a contradiction. Thus $\left \langle a,b \right \rangle =\left \{ a,b,ab,ba \right \}$. $aba=b$ implies $ab=b$, which is a contradiction, so $aba \in \left \{ a,ab,ba \right \}$ and for the same reason, $bab \in \left \{ b,ab,ba \right \}$. Now $aba=ab\: \Leftrightarrow \: bab=ba$, $aba=ba\: \Leftrightarrow \: bab=ab$ and $aba=a\: \Leftrightarrow \: aba \notin \left \{ ab,ba \right \} \: \Leftrightarrow \: bab \notin \left \{ ab,ba \right \}\: \Leftrightarrow \: bab=b$. 
\end{proof}

\vspace{3mm}

\begin{thm} \label{thm2.6}
Every idempotent semigroup $S$ of order $\geq 6$ has a subsemigroup of order $6$.
\end{thm}

\begin{proof}
Suppose that $S$ does not have a subsemigroup of order 6. \\
\\
Case I. $\forall x,y,z \in S\:\: \left | \left \langle x,y,z \right \rangle \right |\leq 5$ \\
By Lemma \ref{lem2.4}, $\forall x,y \in S \:\: \left | \left \langle x,y \right \rangle \right |\leq 3$ and there exist $a,b,c \in S$ such that $\left | \left \langle a,b \right \rangle \right |=3$ and $\left | \left \langle a,b,c \right \rangle \right |=4$. By Lemma \ref{lem2.4a}, there exists $d \in S$ such that $\left | \left \langle a,b,c,d \right \rangle \right |\geq 6$. Then $\left | \left \langle a,b,d \right \rangle \right |\leq 5$ implies $ad,da \in \left \langle a,b \right \rangle \cup \left \{ d,bd,db \right \}$ or $bd,db \in \left \langle a,b \right \rangle \cup \left \{ d,ad,da \right \}$. Without loss of generality, let $\left \langle a,b,d \right \rangle $ be either $\left \{ a,b,ab,d \right \}$, $\left \{ a,b,ab,d,ad \right \}$ or $\left \{ a,b,ab,d,da \right \}$. If $\left \langle a,b,d \right \rangle =\left \{ a,b,ab,d \right \}$, then $\left \langle a,b,c,d \right \rangle=\left \{ a,b,ab \right \}\cup \left \langle c,d \right \rangle$ so $6\leq \left | \left \langle a,b,c,d \right \rangle \right |\leq \left | \left \{ a,b,ab \right \} \right |+\left | \left \langle c,d \right \rangle \right |\leq 6$, a contradiction. Now $\left | \left \langle a,b,c \right \rangle \right |=4$ and Lemma \ref{lem2.5} imply $\left \langle a,b,c,d \right \rangle \in \left \{ a,b,c,d,ab,ad,da,cd,dc \right \}$, so we have the following 4 cases. \\ 
\\
(i) $\left \langle a,b,c,d \right \rangle = \left \{ a,b,ab,c,d,ad,cd \right \}$ : From the table \ref{tab:3}, we get the following results. 

$\cdot$ $\left \langle a,b,c,d \right \rangle-\left \{ a \right \}$ is not a subsemigroup of $S$ : $adc=a$

$\cdot$ $\left \langle a,b,c,d \right \rangle-\left \{ b \right \}$ is not a subsemigroup of $S$ : $dab=b$ or $cdab=b$

$\cdot$ $\left \langle a,b,c,d \right \rangle-\left \{ c \right \}$ is not a subsemigroup of $S$ : $dab=c$ or $cda=c$ or $cdb=c$ or $cdab=c$

$\cdot$ $\left \langle a,b,c,d \right \rangle-\left \{ d \right \}$ is not a subsemigroup of $S$ : $bad=d$ or $bcd=d$ \\
$adc=a$ implies $ac=a$, and both $bad=d$ and $bcd=d$ imply $bd=d$. If $cdab=b$, then $(cda)d=(cda)(bd)=(cdab)d=bd=d$ so $cda \in \left \{ b,d \right \}$, a contradiction. Thus $dab=b$. If $cdb=c$ or $cdab=c$, then $cb=c$ so $ab=acb=ac=a$, a contradiction. Thus $cda=c$. $da \in \left \{ a,d,ad \right \}$ and $cd,cad \neq c$, so $da=a$. This implies $ab=dab=b$, a contradiction. \\
\\
(ii) $\left \langle a,b,c,d \right \rangle = \left \{ a,b,ab,c,d,ad,dc \right \}$ : From the table \ref{tab:4}, we get the following results. 

$\cdot$ $\left \langle a,b,c,d \right \rangle-\left \{ a \right \}$ is not a subsemigroup of $S$ : $adc=a$ or $abdc=a$

$\cdot$ $\left \langle a,b,c,d \right \rangle-\left \{ c \right \}$ is not a subsemigroup of $S$ : $bdc=c$ or $ax=c$ for some $x \in \left \langle a,b,c,d \right \rangle$ \\
Both $adc=a$ and $abcd=a$ imply $ac=a$. If there exists $x \in \left \langle a,b,c,d \right \rangle$ such that $ax=c$, then $a=ac=a(ax)=ax=c$, a contradiction. Thus $bdc=c$, and $cdc=(bdc)dc=b(dcdc)=bdc=c$ implies $cd=c$. Now $ad=(ac)d=a(cd)=ac=a$, a contradiction. \\
\\
(iii) $\left \langle a,b,c,d \right \rangle = \left \{ a,b,ab,c,d,da,cd \right \}$ : From the table \ref{tab:5}, we get the following results.

$\cdot$ $\left \langle a,b,c,d \right \rangle-\left \{ b \right \}$ is not a subsemigroup of $S$ : $dab=b$ or $cab=b$ or $cdab=b$

$\cdot$ $\left \langle a,b,c,d \right \rangle-\left \{ c \right \}$ is not a subsemigroup of $S$ : $cda=c$ or $cdb=c$ or $cdab=c$ \\
Suppose that $cab=b$ or $cdab=b$. Then $cb=b$. If $xb=c$ for some $x \in \left \langle a,b,c,d \right \rangle$, then $b=cb=(xb)b=xb=c$, a contradiction. Thus $cda=c$ and this implies $ca=c$. If $ac=a$, $a=ac=a(cda)=ada$, so $ad=a$. Then $cd=cad=ca=c$, a contradiction. If $ac=c$, then $ab=acb=cb=b$, a contradiction. If $ac=ab$, then $c=cac=cab=cb=b$, a contradiction. So we get $dab=b$ and this implies $db=b$. \\
$(ada)b=a(dab)=ab \neq b$ implies $ada \neq da$, and this implies $ad=a$. $(ba)d=ba$ implies $ba \neq a$ and $d(ba)=ba$ implies $ba \neq ab$, so $ba=b$ and $bd=bad=ba=b$. If $cda=c$, then $cd=(cda)d=cd(ad)=cda=c$, a contradiction. If $yb=c$ for some $y \in \left \langle a,b,c,d \right \rangle$, then $cd=ybd=yb=c$, a contradiction. \\
\begin{table}
\small
\centering
\setlength\tabcolsep{2.5pt}
\begin{tabular}{c"c|c|c|c|c|c|c}
    $*$ & $a$ & $b$ & $ab$ & $c$ & $d$ & $ad$ & $cd$ \\\thickhline
    $a$ & $a$ & $ab$ & $ab$ & $a,c,ab$ & $ad$ & $ad$ & $ab,ad,cd$ \\\hline
    $b$ & $a,b,ab$ & $b$ & $b,ab$ & $b,c,ab$ & $not\;a,c$ & $not\;a,c$ & $not\;a,c$ \\\hline
    $ab$ & $a,ab$ & $ab$ & $ab$ & $a,c,ab$ & $ab,ad,cd$ & $ab,ad,cd$ & $ab,ad,cd$ \\\hline
    $c$ & $a,c,ab$ & $b,c,ab$ & $b,c,ab$ & $c$ & $cd$ & $ab,ad,cd$ & $cd$ \\\hline
    $d$ & $a,d,ad$ & $not \; a,c$ & $not\;a$ & $c,d,cd$ & $d$ & $d,ad$ & $d,cd$ \\\hline
    $ad$ & $a,ad$ & $ab,ad,cd$ & $ab,ad,cd$ & $not\;b,d$ & $ad$ & $ad$ & $ab,ad,cd$ \\\hline
    $cd$ & $not\;b,d$ & $not\;a,d$ & $not\;a,d$ & $c,cd$ & $cd$ & $ab,ad,cd$ & $cd$ 
\end{tabular}
\caption{$\left \langle a,b,c,d \right \rangle = \left \{ a,b,ab,c,d,ad,cd \right \}$}
\label{tab:3} 
\end{table} 
\begin{table}
\small
\centering
\setlength\tabcolsep{2.5pt}
\begin{tabular}{c"c|c|c|c|c|c|c}
    $*$ & $a$ & $b$ & $ab$ & $c$ & $d$ & $ad$ & $dc$ \\\thickhline
    $a$ & $a$ & $ab$ & $ab$ & $a,c,ab$ & $ad$ & $ad$ & $a,c,ad,dc$ \\\hline
    $b$ & $a,b,ab$ & $b$ & $b,ab$ & $b,c,ab$ & $not\;a,c$ & $b,d,ab,ad$ & $not\;a,d$ \\\hline
    $ab$ & $a,ab$ & $ab$ & $ab$ & $a,c,ab$ & $ab,ad,dc$ & $ab,ad,dc$ & $not\;b,d$ \\\hline
    $c$ & $a,c,ab$ & $b,c,ab$ & $b,c,ab$ & $c$ & $c,d,dc$ & $c,ad,dc$ & $c,dc$ \\\hline
    $d$ & $a,d,ad$ & $not\;a,c$ & $not\;a,c$ & $dc$ & $d$ & $d,ad$ & $dc$ \\\hline
    $ad$ & $a,ad$ & $c,ab,ad,dc$ & $c,ab,ad,dc$ & $a,c,ad,dc$ & $ad$ & $ad$ & $a,c,ad,dc$ \\\hline
    $dc$ & $a,ab,ad,dc$ & $not\;a,c$ & $not\;a,c$ & $dc$ & $d,dc$ & $d,ab,ad,dc$ & $dc$
\end{tabular}
\caption{$\left \langle a,b,c,d \right \rangle = \left \{ a,b,ab,c,d,ad,dc \right \}$}
\label{tab:4} 
\end{table} 
\begin{table}
\small
\centering
\setlength\tabcolsep{2.5pt}
\begin{tabular}{c"c|c|c|c|c|c|c}
    $*$ & $a$ & $b$ & $ab$ & $c$ & $d$ & $da$ & $cd$ \\\thickhline
    $a$ & $a$ & $ab$ & $ab$ & $a,c,ab$ & $a,d,da$ & $a,da$ & $a,ab,da,cd$ \\\hline
    $b$ & $a,b,ab$ & $b$ & $b,ab$ & $b,c,ab$ & $not\;a,c$ & $not\;c,d$ & $not\;a,c$ \\\hline
    $ab$ & $a,ab$ & $ab$ & $ab$ & $a,c,ab$ & $d,ab,da,cd$ & $a,ab,da,cd$ & $not\;b,c$ \\\hline
    $c$ & $a,c,ab$ & $b,c,ab$ & $b,c,ab$ & $c$ & $cd$ & $not\;b,d$ & $cd$ \\\hline
    $d$ & $da$ & $not\;a,c$ & $not\;a,c$ & $c,d,cd$ & $d$ & $da$ & $d,cd$ \\\hline
    $da$ & $da$ & $not\;a,c$ & $not\;a,c$ & $c,ab,da,cd$ & $d,da$ & $da$ & $d,ab,da,cd$ \\\hline
    $cd$ & $not\;b,d$ & $not\;a,d$ & $not\;a,d$ & $c,cd$ & $cd$ & $not\;b,d$ & $cd$
\end{tabular}
\caption{$\left \langle a,b,c,d \right \rangle = \left \{ a,b,ab,c,d,da,cd \right \}$}
\label{tab:5} 
\end{table}
\\
(iv) $\left \langle a,b,c,d \right \rangle = \left \{ a,b,ab,c,d,da,dc \right \}$ : symmetric to (ii). \\
\\
Case II. $\exists a,b,c \in S\:\: \left | \left \langle a,b,c \right \rangle \right |>6$ 

(a) $\nexists x,y \in S$ such that $\left \langle x,y \right \rangle=\left \{ x,y,xy,yx \right \}$ and $xyx=x$, $yxy=y$ \\
Suppose that there exist $x,y \in S$ such that $\left | \left \langle x,y \right \rangle \right |=5$. Without loss of generality, let $\left \langle x,y \right \rangle$ be $\left \{ x,y,xy,yx,xyx \right \}$. It is easy to see that $yxy=y$ and $\left \langle xy,yx \right \rangle=\left \{ xy,yx,(xy)(yx),(yx)(xy) \right \}=\left \{ y,xy,yx,xyx \right \}$. In this case, $(xy)(yx)(xy)=xy$ and $(yx)(xy)(yx)=yx$, a contradiction. Thus for every $x,y \in S$, $\left | \left \langle x,y \right \rangle \right |\leq4$. \\
Suppose that for every $x,y \in S$, $\left | \left \langle x,y \right \rangle \right | \leq 3$. If one of $\left | \left \langle a,b \right \rangle \right |,\, \left | \left \langle b,c \right \rangle \right |,\, \left | \left \langle c,a \right \rangle \right |$ is 2, then it is easy to see that $\left | \left \langle a,b,c \right \rangle \right | \leq6$, a contradiction. Thus $\left | \left \langle a,b \right \rangle \right |= \left | \left \langle b,c \right \rangle \right |=\left | \left \langle c,a \right \rangle \right |=3$. Without loss of generality, let $\left \langle a,b \right \rangle$ be $\left \{ a,b,ab \right \}$. Now we have the following 4 cases. \\
\\
(i) $\left \langle b,c \right \rangle=\left \{ b,c,bc \right \},\, \left \langle c,a \right \rangle=\left \{ c,a,ca \right \}$ : In this case, $\left \langle a,b,c \right \rangle=\left \{ a,b,c,ab,bc,ca,abc,bca,cab \right \}$. Let $N_1$ be $\left \{ ab,bc,ca,abc,bca,cab \right \}$. Then from the table \ref{tab:6}, it is easy to see that $N_1$, $N_1\cup \left \{ c \right \}$ and $N_1\cup \left \{ b,c \right \}$ are subsemigroups of $\left \langle a,b,c \right \rangle$. Thus there exists a subsemigroup of $S$ of order 6, a contradiction. \\ 
The only nontrivial part of the table \ref{tab:6} is $abca \neq a$ (and $bcab \neq b$, $cabc \neq c$). Suppose that $abca=a$. If $ba=b$, then $bca=(bab)ca=ba=b$ and $ab=a(bca)=a$, a contradiction. If $ba=ab$, then $ab=ba=b(abca)=abca=a$, a contradiction. Thus $ba=a$ and for the same reason, $ac=a$. If $ab=ca$, then $\left | \left \langle a,b,c \right \rangle \right |=\left | \left \langle a,b,c,ab,bc,abc \right \rangle \right |\leq 6$, a contradiction. Thus $a,ab,ca$ are different elements of $\left \langle a,b,c \right \rangle$. Now $\left | \left \langle ab,ca \right \rangle \right |\leq 3$ implies $cab \in \left \{ a,ab,ca \right \}$. If $cab \in \left \{ a,ab \right \}$, then $ca=c(aba)=(cab)a=a$, a contradiction. If $cab=ca$, then $ab=(aca)b=aca=a$, a contradiction. Thus $abca \neq a$, $bcab \neq b$ and $cabc \neq c$. \\
(ii) $\left \langle b,c \right \rangle=\left \{ b,c,bc \right \},\, \left \langle c,a \right \rangle=\left \{ c,a,ac \right \}$ : In this case, $\left \langle a,b,c \right \rangle=\left \{ a,b,c,ab,bc,ac,abc \right \}$. If $xb=a$ or $xc=a$ for some $x \in S$, then $ab=a$ or $ac=a$, a contradiction. Thus $\left \{ b,c,ab,bc,ac,abc \right \}$ is a subsemigroup of $\left \langle a,b,c \right \rangle$, and this implies that there exists a subsemigroup of $S$ of order 6, a contradiction. \\
(iii) $\left \langle b,c \right \rangle=\left \{ b,c,cb \right \},\, \left \langle c,a \right \rangle=\left \{ c,a,ca \right \}$ : symmetric to (ii). \\
(iv) $\left \langle b,c \right \rangle=\left \{ b,c,cb \right \},\, \left \langle c,a \right \rangle=\left \{ c,a,ac \right \}$ : symmetric to (ii). \\
\begin{table}
      \small
      \centering
      \setlength\tabcolsep{2.5pt}
\begin{tabular}{c"c|c|c|c|c|c|c|c|c}
    $*$ & $a$ & $b$ & $c$ & $ab$ & $bc$ & $ca$ & $abc$ & $bca$ & $cab$ \\\thickhline
    $a$ & $a$ & $ab$ & $a,c,ca$ & $ab$ & $abc$ & $a,ca$ & $abc$ & $abc,bca$ & $ab,cab$\\\hline
    $b$ & $a,b,ab$ & $b$ & $bc$ & $b,ab$ & $bc$ & $bca$ & $bc,abc$ & $bca$ & $bca,cab$\\\hline
    $c$ & $ca$ & $b,c,bc$ & $c$ & $cab$ & $c,bc$ & $ca$ & $cab,abc$ & $ca,bca$ & $cab$\\\hline
    $ab$ & $a,ab$ & $ab$ & $abc$ & $ab$ & $abc$ & $abc,bca$ & $abc$ & $abc,bca$ & $ab,abc,cab$\\\hline
    $bc$ & $bca$ & $b,bc$ & $bc$ & $bca,cab$ & $bc$ & $bca$ & $bc,bca,abc$ & $bca$ & $bca,cab$\\\hline
    $ca$ & $ca$ & $cab$ & $c,ca$ & $cab$ & $cab,abc$ & $ca$ & $cab,abc$ & $ca,cab,bca$ & $cab$\\\hline
    $abc$ & $abc,bca$ & $ab,abc$ & $abc$ & $ab,abc,cab$ & $abc$ & $abc,bca$ & $abc$ & $abc,bca$ & $ab,abc,cab$\\\hline
    $bca$ & $bca$ & $bca,cab$ & $bc,bca$ & $bca,cab$ & $bc,bca,abc$ & $bca$ & $bc,bca,abc$ & $bca$ & $bca,cab$\\\hline
    $cab$ & $ca,cab$ & $cab$ & $cab,abc$ & $cab$ & $cab,abc$ & $ca,cab,bca$ & $cab,abc$ & $ca,cab,bca$ & $cab$
\end{tabular}
    \caption{$\left \langle a,b,c \right \rangle=\left \{ a,b,c,ab,bc,ca,abc,bca,cab \right \}$}
    \label{tab:6} 
\end{table} \\
So there exist $a,b \in S$ such that $\left | \left \langle a,b \right \rangle \right |=4$. By Lemma \ref{lem2.5a} and the assumption given in (a), $\left \langle a,b \right \rangle=\left \{ a,b,ab,ba \right \}$ and $\left \{ aba,bab \right \}=\left \{ ab,ba \right \}$. Without loss of generality, suppose that $aba=ab$ and $bab=ba$. $\left | \left \langle b,c \right \rangle \right |\leq 4$ implies $bcb, cbc \in \left \{ b,c,bc,cb \right \}$. $bcb=c$ implies $bcb=bc$ and $cbc=b$ implies $cbc=cb$, so $bcb \in \left \{ b,bc,cb \right \}$ and $cbc \in \left \{ c,bc,cb \right \}$. Suppose that $bcb \notin \left \{ bc,cb \right \}$. $cbc \in \left \{ bc,cb \right \}$ implies $bcb \in \left \{ bc,cb \right \}$, so $cbc \notin \left \{ bc,cb \right \}$ and $bcb=b$, $cbc=c$. If $bc=cb$, then $bc=bcb \notin \left \{ bc,cb \right \}$, a contradiction. Now $b,c \notin \left \{ bc,cb \right \}$ and $bc \neq cb$ imply $\left | \left \langle b,c \right \rangle \right |=4$, which contradicts to the assumption given in (a). Thus $bcb=bc$ or $bcb=cb$, and by the same argument $aca=ac$ or $aca=ca$. Now we have the following 4 cases. \\
\\
(i) $bcb=bc$, $aca=ac$ : Let $N_2$ be $\left \{ abc,acb,bac,bca,cab,cba \right \}$. From the table \ref{tab:9}, it is easy to see that $N_2,\, N_2\cup \left \{ ac \right \},\, N_2\cup \left \{ ac,ca \right \},\, \cdots , \left \langle a,b,c \right \rangle-\left \{ a \right \}$ are subsemigroups of $\left \langle a,b,c \right \rangle$. Thus there exists a subsemigroup of $S$ of order 6, a contradiction. \\
(ii) $bcb=bc$, $aca=ca$ : For every $x,y,z \in S$ such that $\left \{ x,y,z \right \}=\left \{ a,b,c \right \}$, $\left | \left \langle x,yz \right \rangle \right |\leq 4$ implies $xyzx \in \left \{ x,yz,xyz,yzx \right \}$. $xyzx=yz$ implies $xyzx=xyz$, so $xyzx \in \left \{ x,xyz,yzx \right \}$. Suppose that $xyzx=x$. Then $xyx=xy(xyzx)=xyzx=x$ implies $xy=x$ or $yx=x$, and $xzx=(xyzx)zx=xyzx=x$ implies $xz=x$ or $zx=x$. Each of $xy=xz=x$, $xy=zx=x$, $yx=xz=x$ and $yx=zx=x$ implies $xyz=x$, $zxy=x$, $yxz=x$ and $zyx=x$, so $xyzx=x \in N_2$. Thus $xyzx \in N_2$ and $xyzxy \in \left \{ xyzy,xzy,yxzy,yzxy,zxy,zyxy \right \}\subset N_2$ for all $x,y,z \in S$ such that $\left \{ x,y,z \right \}=\left \{ a,b,c \right \}$. From the table \ref{tab:10}, it is easy to see that $N_2,\, N_2\cup \left \{ ac \right \},\, N_2\cup \left \{ ac,ca \right \},\, \cdots , \left \langle a,b,c \right \rangle-\left \{ a \right \}$ are subsemigroups of $\left \langle a,b,c \right \rangle$. Thus there exists a subsemigroup of $S$ of order 6, a contradiction. \\
(iii) $bcb=cb$, $aca=ac$ : symmetric to (ii). \\
(iv) $bcb=cb$, $aca=ca$ : symmetric to (ii). \\
\begin{table}
\small
\centering
\setlength\tabcolsep{2.5pt}
\begin{tabular}{c"c|c|c|c|c|c|c|c|c|c|c|c|c|c|c}
    $*$ & $a$ & $b$ & $c$ & $ab$ & $ba$ & $bc$ & $cb$ & $ca$ & $ac$ & $abc$ & $acb$ & $bac$ & $bca$ & $cab$ & $cba$\\\thickhline
    $a$ & $a$ & $ab$ & $ac$ & $ab$ & $ab$ & $abc$ & $acb$ & $ac$ & $ac$ & $abc$ & $acb$ & $abc$ & $abc$ & $acb$ & $acb$\\\hline
    $b$ & $ba$ & $b$ & $bc$ & $ba$ & $ba$ & $bc$ & $bc$ & $bca$ & $bac$ & $bac$ & $bac$ & $bac$ & $bca$ & $bca$ & $bca$\\\hline
    $c$ & $ca$ & $cb$ & $c$ & $cab$ & $cba$ & $cb$ & $cb$ & $ca$ & $ca$ & $cab$ & $cab$ & $cba$ & $cba$ & $cab$ & $cba$\\\hline
    $ab$ & $ab$ & $ab$ & $abc$ & $ab$ & $ab$ & $abc$ & $abc$ & $abc$ & $abc$ & $abc$ & $abc$ & $abc$ & $abc$ & $abc$ & $abc$ \\\hline
    $ba$ & $ba$ & $ba$ & $bac$ & $ba$ & $ba$ & $bac$ & $bac$ & $bac$ & $bac$ & $bac$ & $bac$ & $bac$ & $bac$ & $bac$ & $bac$\\\hline
    $bc$ & $bca$ & $bc$ & $bc$ & $bca$ & $bca$ & $bc$ & $bc$ & $bca$ & $bca$ & $bca$ & $bca$ & $bca$ & $bca$ & $bca$ & $bca$\\\hline
    $cb$ & $cba$ & $cb$ & $cb$ & $cba$ & $cba$ & $cb$ & $cb$ & $cba$ & $cba$ & $cba$ & $cba$ & $cba$ & $cba$ & $cba$ & $cba$\\\hline
    $ca$ & $ca$ & $cab$ & $ca$ & $cab$ & $cab$ & $cab$ & $cab$ & $ca$ & $ca$ & $cab$ & $cab$ & $cab$ & $cab$ & $cab$ & $cab$\\\hline
    $ac$ & $ac$ & $acb$ & $ac$ & $acb$ & $acb$ & $acb$ & $acb$ & $ac$ & $ac$ & $acb$ & $acb$ & $acb$ & $acb$ & $acb$ & $acb$\\\hline
    $abc$ & $abc$ & $abc$ & $abc$ & $abc$ & $abc$ & $abc$ & $abc$ & $abc$ & $abc$ & $abc$ & $abc$ & $abc$ & $abc$ & $abc$ & $abc$\\\hline
    $acb$ & $acb$ & $acb$ & $acb$ & $acb$ & $acb$ & $acb$ & $acb$ & $acb$ & $acb$ & $acb$ & $acb$ & $acb$ & $acb$ & $acb$ & $acb$\\\hline
    $bac$ & $bac$ & $bac$ & $bac$ & $bac$ & $bac$ & $bac$ & $bac$ & $bac$ & $bac$ & $bac$ & $bac$ & $bac$ & $bac$ & $bac$ & $bac$\\\hline
    $bca$ & $bca$ & $bca$ & $bca$ & $bca$ & $bca$ & $bca$ & $bca$ & $bca$ & $bca$ & $bca$ & $bca$ & $bca$ & $bca$ & $bca$ & $bca$\\\hline
    $cab$ & $cab$ & $cab$ & $cab$ & $cab$ & $cab$ & $cab$ & $cab$ & $cab$ & $cab$ & $cab$ & $cab$ & $cab$ & $cab$ & $cab$ & $cab$\\\hline
    $cba$ & $cba$ & $cba$ & $cba$ & $cba$ & $cba$ & $cba$ & $cba$ & $cba$ & $cba$ & $cba$ & $cba$ & $cba$ & $cba$ & $cba$ & $cba$\\\hline
\end{tabular}
    \caption{$\left \langle a,b,c \right \rangle$ ($bcb=bc$, $aca=ac$)}
    \label{tab:9} 
\end{table}

\begin{table}
      \small
      \centering
      \setlength\tabcolsep{2.5pt}
\begin{tabular}{c"c|c|c|c|c|c|c|c|c|c|c|c|c|c|c}
    $*$ & $a$ & $b$ & $c$ & $ab$ & $ba$ & $bc$ & $cb$ & $ca$ & $ac$ & $abc$ & $acb$ & $bac$ & $bca$ & $cab$ & $cba$\\\thickhline
    $a$ & $a$ & $ab$ & $ac$ & $ab$ & $ab$ & $abc$ & $acb$ & $ca$ & $ac$ & $abc$ & $acb$ & $abc$ & $abca$ & $cab$ & $acba$\\\hline
    $b$ & $ba$ & $b$ & $bc$ & $ba$ & $ba$ & $bc$ & $bc$ & $bca$ & $bac$ & $bac$ & $bacb$ & $bac$ & $bca$ & $bcab$ & $bca$\\\hline
    $c$ & $ca$ & $cb$ & $c$ & $cab$ & $cba$ & $cb$ & $cb$ & $ca$ & $ac$ & $cabc$ & $acb$ & $cbac$ & $cbca$ & $cab$ & $cba$\\\hline
    $ab$ & $ab$ & $ab$ & $abc$ & $ab$ & $ab$ & $abc$ & $abc$ & $abca$ & $abc$ & $abc$ & $abc$ & $abc$ & $abca$ & $abcab$ & $abca$\\\hline
    $ba$ & $ba$ & $ba$ & $bac$ & $ba$ & $ba$ & $bac$ & $bacb$ & $bca$ & $bac$ & $bac$ & $bacb$ & $bac$ & $bacb$ & $bcab$ & $bacba$\\\hline
    $bc$ & $bca$ & $bc$ & $bc$ & $bcab$ & $bca$ & $bc$ & $bc$ & $bca$ & $bac$ & $bcabc$ & $bacb$ & $bac$ & $bca$ & $bcab$ & $bca$\\\hline
    $cb$ & $cba$ & $cb$ & $cb$ & $cba$ & $cba$ & $cb$ & $cb$ & $cba$ & $cbac$ & $cbac$ & $cbacb$ & $cbac$ & $cba$ & $cba$ & $cba$\\\hline
    $ca$ & $ca$ & $cab$ & $ac$ & $cab$ & $cab$ & $cabc$ & $acb$ & $ca$ & $ac$ & $cabc$ & $acb$ & $cabc$ & $cabca$ & $cab$ & $acba$\\\hline
    $ac$ & $ca$ & $acb$ & $ac$ & $cab$ & $acba$ & $acb$ & $acb$ & $ca$ & $ac$ & $cabc$ & $acb$ & $acbac$ & $acba$ & $cab$ & $acba$\\\hline
    $abc$ & $abca$ & $abc$ & $abc$ & $abcab$ & $abca$ & $abc$ & $abc$ & $abca$ & $abc$ & $abc$ & $abc$ & $abc$ & $abca$ & $abcab$ & $abca$\\\hline
    $acb$ & $acba$ & $acb$ & $acb$ & $acba$ & $acba$ & $acb$ & $acb$ & $acba$ & $acbac$ & $acbac$ & $acb$ & $acbac$ & $acba$ & $acba$ & $acba$\\\hline
    $bac$ & $bca$ & $bacb$ & $bac$ & $bcab$ & $bacba$ & $bacb$ & $bacb$ & $bca$ & $bac$ & $bcabc$ & $bacb$ & $bac$ & $bacba$ & $bcab$ & $bacba$\\\hline
    $bca$ & $bca$ & $bcab$ & $bac$ & $bcab$ & $bcab$ & $bcabc$ & $bacb$ & $bca$ & $bac$ & $bcabc$ & $bacb$ & $bcabc$ & $bca$ & $bcab$ & $bacba$\\\hline
    $cab$ & $cab$ & $cab$ & $cabc$ & $cab$ & $cab$ & $cabc$ & $cabc$ & $cabca$ & $cabc$ & $cabc$ & $cabc$ & $cabc$ & $cabca$ & $cab$ & $cabca$\\\hline
    $cba$ & $cba$ & $cba$ & $cbac$ & $cba$ & $cba$ & $cbac$ & $cbacb$ & $cba$ & $cbac$ & $cbac$ & $cbacb$ & $cbac$ & $cba$ & $cba$ & $cba$\\\hline
\end{tabular}
    \caption{$\left \langle a,b,c \right \rangle$ ($bcb=bc$, $aca=ca$)}
    \label{tab:10} 
\end{table}

(b) $\exists x,y \in S$ such that $\left \langle x,y \right \rangle=\left \{ x,y,xy,yx \right \}$ and $xyx=x$, $yxy=y$ \\
Let $a$ and $b$ be two elements of $S$ such that $\left \langle a,b \right \rangle=\left \{ a,b,ab,ba \right \}$ and $aba=a, bab=b$. By Lemma \ref{lem2.4a}, there exists $c \in S$ such that $\left | \left \langle a,b,c \right \rangle \right |\geq 6$. Let $x$, $y$, $z$ and $w$ be $bcb$, $cb$, $bca$ and $bc$. Then $\left \langle a,b,x \right \rangle=\left \{ a,b,ab,ba,x,ax,xa,axa \right \}$. Let $M_1$ be $\left \{ x,ax,xa,axa \right \}$. From the table \ref{tab:11}, $M_1\cup \left \{ a,ab \right \}$ and $M_1\cup \left \{ b,ba \right \}$ are subsemigroups of $\left \langle a,b,x \right \rangle$. If $\left | M_1 \right |=4$, then $\left | M_1\cup \left \{ a,ab \right \} \right | \neq 6$ implies $\left | M_1\cup \left \{ a,ab \right \} \right | \leq 5$ and $\left | M_1\cup \left \{ b,ba \right \} \right | \neq 6$ implies $\left | M_1\cup \left \{ b,ba \right \} \right | \leq 5$, so $\left | \left \langle a,b,x \right \rangle \right | \leq6$. It is easy to see that $M_1$ does not have a subsemigroup of order 3, so $\left | M_1 \right |<4$ implies $\left | M_1 \right | \leq2$ and $\left | \left \langle a,b,x \right \rangle \right | \leq6$. Now we have the following 2 cases. \\
\\
\begin{table}
    \begin{minipage}{.5\linewidth}
      \small
      \centering
      \setlength\tabcolsep{2.5pt}
\begin{tabular}{c"c|c|c|c|c|c|c|c}
    $*$ & $a$ & $b$ & $ab$ & $ba$ & $x$ & $ax$ & $xa$ & $axa$ \\\thickhline
    $a$ & $a$ & $ab$ & $ab$ & $a$ & $ax$ & $ax$ & $axa$ & $axa$\\\hline
    $b$ & $ba$ & $b$ & $b$ & $ba$ & $x$ & $x$ & $xa$ & $xa$\\\hline
    $ab$ & $a$ & $ab$ & $ab$ & $a$ & $ax$ & $ax$ & $axa$ & $axa$\\\hline
    $ba$ & $ba$ & $b$ & $b$ & $ba$ & $x$ & $x$ & $xa$ & $xa$\\\hline
    $x$ & $xa$ & $x$ & $x$ & $xa$ & $x$ & $x$ & $xa$ & $xa$\\\hline
    $ax$ & $axa$ & $ax$ & $ax$ & $axa$ & $ax$ & $ax$ & $axa$ & $axa$\\\hline
    $xa$ & $xa$ & $x$ & $x$ & $xa$ & $x$ & $x$ & $xa$ & $xa$\\\hline
    $axa$ & $axa$ & $ax$ & $ax$ & $axa$ & $ax$ & $ax$ & $axa$ & $axa$
\end{tabular}
      \caption{$\left \langle a,b,x \right \rangle$ ($x=bcb$)}
      \label{tab:11} 
    \end{minipage}%
    \begin{minipage}{.5\linewidth}
      \small
      \centering
      \setlength\tabcolsep{2.5pt}
\begin{tabular}{c"c|c|c|c|c|c|c|c}
    $*$ & $a$ & $b$ & $ab$ & $ba$ & $y$ & $by$ & $ay$ & $bay$ \\\thickhline
    $a$ & $a$ & $ab$ & $ab$ & $a$ & $ay$ & $by$ & $ay$ & $ay$\\\hline
    $b$ & $ba$ & $b$ & $b$ & $ba$ & $by$ & $by$ & $bay$ & $bay$\\\hline
    $ab$ & $a$ & $ab$ & $ab$ & $a$ & $by$ & $by$ & $ay$ & $ay$\\\hline
    $ba$ & $ba$ & $b$ & $b$ & $ba$ & $bay$ & $by$ & $bay$ & $bay$\\\hline
    $y$ & $y$ & $y$ & $y$ & $y$ & $y$ & $y$ & $y$ & $y$\\\hline
    $by$ & $by$ & $by$ & $by$ & $by$ & $by$ & $by$ & $by$ & $by$\\\hline
    $ay$ & $ay$ & $ay$ & $ay$ & $ay$ & $ay$ & $ay$ & $ay$ & $ay$\\\hline
    $bay$ & $bay$ & $bay$ & $bay$ & $bay$ & $bay$ & $bay$ & $bay$ & $bay$
\end{tabular}
    \caption{$\left \langle a,b,y \right \rangle$ ($y=cb$)}
    \label{tab:12} 
    \end{minipage} 
\end{table}
\begin{table}
    \begin{minipage}{.5\linewidth}
      \small
      \centering
      \setlength\tabcolsep{2.5pt}
\begin{tabular}{c"c|c|c|c|c|c|c|c}
    $*$ & $a$ & $b$ & $ab$ & $ba$ & $z$ & $az$ & $zb$ & $azb$ \\\thickhline
    $a$ & $a$ & $ab$ & $ab$ & $a$ & $az$ & $az$ & $azb$ & $azb$\\\hline
    $b$ & $ba$ & $b$ & $b$ & $ba$ & $z$ & $z$ & $zb$ & $zb$\\\hline
    $ab$ & $a$ & $ab$ & $ab$ & $a$ & $az$ & $az$ & $azb$ & $azb$\\\hline
    $ba$ & $ba$ & $b$ & $b$ & $ba$ & $z$ & $z$ & $zb$ & $zb$\\\hline
    $z$ & $z$ & $zb$ & $zb$ & $z$ & $z$ & $z$ & $zb$ & $zb$\\\hline
    $az$ & $az$ & $azb$ & $azb$ & $az$ & $az$ & $az$ & $azb$ & $azb$\\\hline
    $zb$ & $z$ & $zb$ & $zb$ & $z$ & $z$ & $z$ & $zb$ & $zb$\\\hline
    $azb$ & $az$ & $azb$ & $azb$ & $az$ & $az$ & $az$ & $azb$ & $azb$
\end{tabular}
      \caption{$\left \langle a,b,z \right \rangle$ ($z=bca$)}
      \label{tab:13} 
    \end{minipage}%
    \begin{minipage}{.5\linewidth}
      \small 
      \centering
      \setlength\tabcolsep{2.5pt}
\begin{tabular}{c"c|c|c|c|c|c|c|c}
    $*$ & $a$ & $b$ & $ab$ & $ba$ & $w$ & $aw$ & $wa$ & $waw$ \\\thickhline
    $a$ & $a$ & $ab$ & $ab$ & $a$ & $aw$ & $aw$ & $a,wa$ & $aw$\\\hline
    $b$ & $ba$ & $b$ & $b$ & $ba$ & $w$ & $w$ & $wa$ & $waw$\\\hline
    $ab$ & $a$ & $ab$ & $ab$ & $a$ & $aw$ & $aw$ & $a,wa$ & $aw$\\\hline
    $ba$ & $ba$ & $b$ & $b$ & $ba$ & $w$ & $w$ & $wa$ & $waw$\\\hline
    $w$ & $wa$ & $b$ & $b,wa$ & $wa$ & $w$ & $waw$ & $wa$ & $waw$\\\hline
    $aw$ & $a,wa$ & $ab$ & $ab,wa$ & $a$ & $aw$ & $aw$ & $a,wa$ & $aw$\\\hline
    $wa$ & $wa$ & $b,wa$ & $b,wa$ & $wa$ & $waw$ & $waw$ & $wa$ & $waw$\\\hline
    $waw$ & $wa$ & $b,wa$ & $b,wa$ & $wa$ & $waw$ & $waw$ & $wa$ & $waw$
\end{tabular}
    \caption{$\left \langle a,b,w \right \rangle$ ($w=bc$)}
    \label{tab:14} 
    \end{minipage} 
\end{table}
(i) $x \notin \left \langle a,b \right \rangle$ : $\left | \left \langle a,b,x \right \rangle \right | \leq5$ implies $ax,xa \in \left \langle a,b \right \rangle \cup \left \{ x \right \}$. Each of $b(ax)=x$ and $(xa)b=x$ implies $ax \notin \left \langle a,b \right \rangle$ and $xa \notin \left \langle a,b \right \rangle$, so $ax=xa=x$ and $abc=axc=xc=bc$, $cba=cxa=cx=cb$. Thus $\left | \left \langle a,b,y \right \rangle \right |=\left \{ a,b,ab,ba,y,by,ay,bay \right \}$. From the table \ref{tab:12}, $\left \langle a,b,y \right \rangle-\left \{ y,by \right \}$ and $\left \langle a,b,y \right \rangle-\left \{ y \right \}$ are subsemigroups of $\left \langle a,b,y \right \rangle$, so $\left | \left \langle a,b,y \right \rangle \right |\leq 5$. $ya=yb=y$ implies $y \notin \left \langle a,b \right \rangle$, so $\left \langle a,b,y \right \rangle=\left \langle a,b \right \rangle\cup \left \{ y \right \}$ and $ay,by \in \left \langle a,b \right \rangle\cup \left \{ y \right \}$. $by=x \notin \left \langle a,b \right \rangle$ and $(ay)a=(ay)b=ay$ implies $ay \notin \left \langle a,b \right \rangle$, so $ay=by=y$. For the same reason, $wa=wb=w$. \\
(ii) $x \in \left \langle a,b \right \rangle$ : $bx=xb=x$ implies $x=b$. Let $M_2$ be $\left \{ z,az,zb,azb \right \}$. Then $\left \langle a,z,b \right \rangle=\left \langle a,b \right \rangle\cup M_2$. From the table \ref{tab:13}, $M_2\cup \left \{ a,ab \right \}$ and $M_2\cup \left \{ b,ba \right \}$ are subsemigroups of $\left \langle a,b,z \right \rangle$. If $\left | M_2 \right |=4$, then $\left | M_2\cup \left \{ a,ab \right \} \right | \neq 6$ implies $\left | M_2\cup \left \{ a,ab \right \} \right | \leq 5$, and $\left | M_2\cup \left \{ b,ba \right \} \right | \neq 6$ implies $\left | M_2\cup \left \{ b,ba \right \} \right | \leq 5$, so $\left | \left \langle a,b,z \right \rangle \right | \leq6$. It is easy to see that $M_2$ does not have a subsemigroup of order 3, so $\left | M_2 \right |<4$ implies $\left | M_2 \right | \leq2$ and $\left | \left \langle a,b,z \right \rangle \right | \leq6$. $bz=za=z$ implies $z \notin \left \{ a,b,ab \right \}$, so $z=ba$ or $z \notin \left \langle a,b \right \rangle$. If $z \notin \left \langle a,b \right \rangle$, then $\left | \left \langle a,b,z \right \rangle \right | \leq5$ implies $az,zb \in \left \langle a,b \right \rangle\cup \left \{ z \right \}$. $a(az)=(az)a=az$ implies $az \notin \left \{ b,ab,ba \right \}$ and $b(zb)=(zb)b=zb$ implies $zb \notin \left \{ a,ab,ba \right \}$. Each of $az=a$ and $zb=b$ implies $z=ba$. Thus if $z \neq ba$, then $az=zb=z$. \\
$x=b$ and [$z=ba$ or $az=zb=z$] implies $\left \langle a,b,w \right \rangle=\left \{ a,b,ab,ba,w,aw,wa,waw \right \}$. Let $M_3$ be $\left \{ w,aw,wa,waw \right \}$. Then from the table \ref{tab:14}, $M_3\cup \left \{ a,ba \right \}$ is a subsemigroup of $\left \langle a,b,w \right \rangle$, so $\left | \left \langle M_3\cup \left \{ a,ba \right \} \right \rangle \right |\leq 5$. If $\left | M_3 \right |=4$, then $a \in M_3$ or $ba \in M_3$. $ba \neq a$ implies $a \notin \left \{ w,wa,waw \right \}$. If $a=aw$, then $wa=waw$, so $\left | \left \langle a,b,w \right \rangle \right | \leq6$. $b \neq ab,ba$ implies $b \notin \left \{ aw,wa \right \}$. If $b=w$ or $b=waw$, then $ab=aw$, so $\left | \left \langle a,b,w \right \rangle \right | \leq6$. Suppose that $\left | M_3 \right | \leq3$. Each of $aw=w$, $wa=w$ and $waw=aw$ implies $waw=w$. $aw=wa$ implies $waw=aw$, and $waw=wa$ implies $aw=awa \in \left \{ a,wa \right \}$. If $waw=w$, then $ba=wba=(waw)ba=wa$. Thus $\left | \left \langle a,b,w \right \rangle \right | \leq6$. \\
If $w \in \left \langle a,b \right \rangle$, then $wb=b$ implies $w=b$ or $w=ba$. If $w \notin \left \langle a,b \right \rangle$, then $\left | \left \langle a,b,w \right \rangle \right | \leq5$ implies $aw \in \left \langle a,b \right \rangle\cup \left \{ w \right \}$. $(aw)b=ab$ implies $aw \notin \left \{ b,ba,w \right \}$, so $aw \in \left \{ a,ab \right \}$ and $w=b(aw) \in \left \{ b,ba \right \}$. \\
\\
By (i) and (ii), $acb=bca=bc=cb$ or $bc \in \left \{ b,ba \right \}$ holds. The assumption $aba=a, bab=b$ is symmetric in $a$ and $b$, so $bca=acb=ac=ca$ or $ac \in \left \{ a,ab \right \}$ holds. Now we have the following 4 cases. \\
\\
(i) $acb=bca=bc=cb=ac=ca$ : $\left | \left \langle a,b,c \right \rangle \right |=\left | \left \{ a,b,ab,ba,c,ac \right \} \right |\leq 6$, a contradiction. \\
(ii) $acb=bca=bc=cb$ and $ac \in \left \{ a,ab \right \}$ : $ac \in \left \{ a,ab \right \}$ implies $acb=ab$, so $ab=bca$. Then $b=bab=ab$, a contradiction. \\
(iii) $bca=acb=ac=ca$ and $bc \in \left \{ b,ba \right \}$ : symmetric to (ii). \\
(iv) $bc \in \left \{ b,ba \right \}$ and $ac \in \left \{ a,ab \right \}$ : $ac \in \left \{ a,ab \right \}$ implies $cac \in \left \{ ca,cab \right \}$, so $\left \langle a,b,cac \right \rangle=\left \{ a,b,ab,ba,ca,cab \right \}$ and $\left | \left \langle a,b,cac \right \rangle \right |\leq 5$. If $cab=ca$, then $ab=(aca)b=a(cab)=aca=a$, a contradiction. Thus $ca \in \left \langle a,b \right \rangle$ or $cab \in \left \langle a,b \right \rangle$. If $cab \in \left \langle a,b \right \rangle$, then $(cab)a=ca$ implies $ca \in \left \langle a,b \right \rangle$. Thus $ca \in \left \langle a,b \right \rangle$ and for the same reason, $cb \in \left \langle a,b \right \rangle$. Now $bc,ac,cb,ca \in \left \langle a,b \right \rangle$ so $\left | \left \langle a,b,c \right \rangle \right |=\left | \left \{ a,b,ab,ba,c \right \} \right |\leq 5$, a contradiction.  
\end{proof}

\vspace{3mm}

\begin{lem} \label{lem2.8}
For every positive integer $n \neq 1,2,4,6,12$, there exist positive integers $p$ and $q$ such that 
\begin{equation} \label{eq:211}
max \left \{ (p-1)q,\, p(q-1) \right \}<n<pq.
\end{equation}
\end{lem}

\begin{proof}
Suppose that there exists a positive integer $n \geq 15$ such that for every positive integer $k \leq \sqrt{n+1}$, $k$ divides $n$. If $n \geq 15$, then there uniquely exists a positive integer $t \geq 4$ such that $t^2\leq n+1< (t+1)^2$. By the assumption, $t \mid n$, so $n=t^2$ or $n=t(t+1)$. In each case $t-1 \mid n$ implies $t-1 \mid 1$ and $t-1 \mid 2$, which contradicts to $t \geq 4$. Thus for every positive integer $n \geq 15$, there exists a positive integer $k \leq \sqrt{n+1}$ such that $k$ does not divide $n$. For $m=\left \lfloor \frac{n}{k} \right \rfloor+1$, it is easy to see that $m(k-1)\leq (m-1)k<n<mk$, so $(p,q)=(k,m)$ satisfies \eqref{eq:211} for every $n \geq 15$. If $n \geq 3$ is odd, then $(p,q)=(2,\frac{n+1}{2})$ satisfies \eqref{eq:211}. If $n$ is 8, 10 or 14, $(p,q)=(3,3),(3,4)$ and $(3,5)$ satisfy \eqref{eq:211} for each $n$.
\end{proof}

\vspace{3mm}

\begin{thm} \label{thm2.9}
For every positive integer $n \neq 1,2,4,6$, there exists an idempotent subsemigroup of order $\geq n$ which does not have a subsemigroup of order $n$.
\end{thm}

\begin{proof}
Let $S_{p,q} \: (p,q \in \mathbb{N})$ denote the set $\mathbb{Z}_p \times \mathbb{Z}_q$ with a binary operation $(a_1,b_1)\cdot (a_2,b_2)=(a_1,b_2)$, which is clearly an idempotent semigroup. By Lemma \ref{lem2.8}, for every positive integer $n \neq 1,2,4,6,12$, there exist $p,q \in \mathbb{N}$ such that \eqref{eq:211} holds. For any subsemigroup $S$ of $S_{p,q}$, let $A_S$ and $B_S$ be $\left \{ a \in \mathbb{Z}_p \mid \exists b \in \mathbb{Z}_q \:\:(a,b)\in S  \right \}$ and $\left \{ b \in \mathbb{Z}_q \mid \exists a \in \mathbb{Z}_p \:\:(a,b)\in S  \right \}$. If $A_S\neq \mathbb{Z}_p$ or $B_S\neq \mathbb{Z}_q$, then $\left | S \right |\leq max \left \{ (p-1)q,\, p(q-1) \right \}<n$. If $A_S=\mathbb{Z}_p$ and $B_S= \mathbb{Z}_q$, then for every $(a,b) \in S_{p,q}$ there exist $a' \in \mathbb{Z}_p$ and $b' \in \mathbb{Z}_q$ such that $(a,b'),(a',b)\in S$. Thus $(a,b)=(a,b')\cdot (a',b)\in S$, so $\left | S \right |=\left | S_{p,q} \right |=pq>n$. Now the only remaining case is $n=12$. The set $S_{3,3} \sqcup S_{2,2}$ with a binary operation \[x\ast y=\left\{\begin{matrix}
x\cdot y\: \: (x,y \in S_{3,3}\: \:  or\: \:  x,y\in S_{2,2})\\ 
x \:\: (x \in S_{2,2}\: \:  and\: \:  y\in S_{3,3}) \\
y \:\: (x \in S_{3,3}\: \:  and\: \:  y\in S_{2,2})
\end{matrix}\right.\]
is an idempotent semigroup of order 13 which does not have a subsemigroup of order 12. 
\end{proof} 

\vspace{3mm}

\section{Case II : $r>2$}
The idempotency of semigroup $S$ implies $x^r=x$ for every $r>2$ and $x \in S$. From the results of previous section, for every positive integer $r>2$ and $n \neq 1,2,4,6$, there exists a semigroup of order $\geq n$ in which $x^r=x$ for every $x \in S$ and does not have a subsemigroup of order $n$. For every semigroup $S$ in which $x^r=x$, $\forall a \in S$ satisfies $(a^{r-1})^2=a^{r-1}$. Thus there exists a subsemigroup of order 1, namely $\left \{ a^{r-1} \right \}$. 

\vspace{3mm}

\begin{lem} \label{lem3.1}
For every positive integer $r>2$, there exists a semigroup of order $\geq 6$ in which $x^r=x$ and does not have a subsemigroup of order $6$. 
\end{lem}

\begin{proof}
Let $q$ be an arbitrary prime factor of $r-1$. Then an abelian group $\mathbb{Z}_q\oplus \mathbb{Z}_q\oplus \mathbb{Z}_q$ is a semigroup of order $q^3>6$ in which $x^r=x$. The order of its subsemigroup divides $q^3$, so it cannot be 6. 
\end{proof}

\vspace{3mm}

\begin{lem} \label{lem3.2}
For every positive integer $r$ in which $r-1$ has a prime factor greater than $2$, there exists a semigroup of order $\geq 4$ in which $x^r=x$ and does not have a subsemigroup of order $2$ or $4$. 
\end{lem}

\begin{proof}
Let $q$ be a prime factor of $r-1$ greater than 2. Then an abelian group $\mathbb{Z}_q\oplus \mathbb{Z}_q$ is a semigroup of order $q^2>4$ in which $x^r=x$. The order of its subsemigroup divides $q^2$, so it cannot be 2 or 4. 
\end{proof}

\vspace{3mm}

\begin{thm} \label{thm3.3}
For a positive integer of the form $r=2^m+1 \:\:(m \in \mathbb{N})$, every semigroup $S$ of order $\geq2$ in which $x^r=x$ has a subsemigroup of order $2$. 
\end{thm}

\begin{proof}
If $S$ is idempotent, then by Theorem \ref{thm2.1}, $S$ has a subsemigroup of order 2. Suppose that there exists $a \in S$ such that $a^2 \neq a$. Let $k$ be the smallest positive integer $t>2$ such that $a^t=a$. Then there uniquely exist nonnegative integers $q$ and $r \leq k-2$ such that $2^m=(k-1)q+r$. If $r>0$, $a=a^{2^m+1}=a^{(k-1)q+r+1}=a^{(k-1)(q-1)+r+1}=\cdots=a^{r+1}$ and $1<r+1\leq k-1$, a contradiction. Thus $r=0$ and $k=2^{m'}+1$ for some positive integer $m'$. Now it is easy to see that $\left \{ a^{2^{m'-1}},a^{2^{m'}} \right \}$ is a subsemigroup of $S$ of order 2.
\end{proof}

\vspace{3mm}

\begin{thm} \label{thm3.4}
For a positive integer of the form $r=2^m+1 \:\:(m \in \mathbb{N})$, every semigroup $S$ of order $\geq4$ in which $x^r=x$ has a subsemigroup of order $4$. 
\end{thm}

\begin{proof}
Suppose that $S$ does not have a subsemigroup of order 4. By Theorem \ref{thm2.3}, $S$ is not an idempotent semigroup. Suppose that there exists $a \in S$ such that $a^3 \neq a$. Let $k$ be the smallest positive integer $t>3$ such that $a^t=a$. From the proof of Theorem \ref{thm3.3}, $k=2^{m'}+1$ for some positive integer $m'\geq 2$. It is easy to see that $\left \{ a^{2^{m'-2}},a^{2^{m'-1}},a^{3\cdot 2^{m'-2}},a^{2^{m'}} \right \}$ is a subsemigroup of $S$ of order 4, a contradiction. Thus every element $x$ of $S$ satisfies $x^3=x$ and there exists $a \in S$ such that $a^2 \neq a$. \\
Suppose that for every $x \in S-\left \langle a \right \rangle$, $\left | \left \langle a,x \right \rangle \right |=3$. Then $x \in S-\left \langle a \right \rangle$ implies $x^2 \in \left \{ a,a^2,x \right \}$ and $(x^2)^2=x^2$, so $x^2=a^2$ or $x^2=x$. Let $b$ and $c$ be two different elements of $S-\left \langle a \right \rangle$. We have the following 3 cases. \\
\\
(i) $b^2=a^2$, $c^2=a^2$ : $(bc)(cb)=bb^2b=b^2=a^2$, so $\left \{ a,a^2,bc,cb \right \}$ is a subsemigroup of $S$ and $\left | \left \{ a,a^2,bc,cb \right \} \right |\leq 3$. If $bc \in \left \{ a,a^2 \right \}$, then $b=b^3=bc^2 \in \left \{ ac,a^2c \right \}$, a contradiction. Thus $bc \notin \left \{ a,a^2 \right \}$ and for the same reason, $cb \notin \left \{ a,a^2 \right \}$. This implies $bc=cb$, so $bcb=cb^2=c^3=c$ and $cbc=b$. Thus $\left \langle a,b,c \right \rangle=\left \{ a,a^2,b,c,bc \right \}$. For $x \in \left \{ b,c,bc \right \}$, $x^2=a^2$ so $a^2x=xa^2=x^3=x$. Thus $\left \{ a^2,b,c,bc \right \}$ is a subsemigroup of $\left \langle a,b,c \right \rangle$, a contradiction. \\
(ii) $c^2=c$ : Let $M_1$ be $\left \{ a,a^2,c,cbc \right \}$. Then $bcbc \in \left \{ a^2,bc \right \}$ and $cbcbc \in \left \{ ca^2,cbc \right \} \subset M_1$. This implies $M_1$ and $M_1 \cup \left \{ bc \right \}$ are subsemigroups of $\left \langle a,b,c \right \rangle$. $\left | M_1 \right | \neq 4$ implies $\left | M_1 \right | \leq 3$, and $\left | M_1 \cup \left \{ bc \right \}\right | \neq 4$ implies $\left | M_1 \cup \left \{ bc \right \} \right | \leq 3$. Thus $bc \in \left \{ a,a^2,c \right \}$ and for the same reason, $cb \in \left \{ a,a^2,c \right \}$. Then $\left | \left \langle a,b,c \right \rangle \right |=\left | \left \{ a,a^2,b,c \right \} \right |=4$, a contradiction. \\
(iii) $b^2=b$ : symmetric to (ii). \\
\\
Thus there exists $u \in S-\left \langle a \right \rangle$ such that $\left | \left \langle a,u \right \rangle \right | \geq4$. Let $x$, $y$ and $z$ be $a^2$, $au$ and $a^2u$. Then $\left \langle x,y \right \rangle=\left \{ x,y,y^2,yx,y^2x \right \}$ and $\left \langle y,yx \right \rangle=\left \{ y,y^2,yx,y^2x \right \}$. $\left | \left \langle y,yx \right \rangle \right | \neq 4$ implies $\left | \left \langle y,yx \right \rangle \right | \leq3$, and $\left | \left \langle x,y \right \rangle \right | \neq 4$ implies $\left | \left \langle x,y \right \rangle \right | \leq3$. Each of $au=a^2$, $aua^2=a^2$ and $auau=aua^2$ implies $auau=au$, and $auau=a^2$ implies $aua^2=au$, so $auau=au$ or $aua^2=au$. For the same reason, $uaua=ua$ or $a^2ua=ua$. $auau=au$ implies $uaua=u(auau)a=ua$, and $uaua=ua$ implies $auau=a(uaua)u=au$, so $auau=au$ is equivalent to $uaua=ua$. If $auau \neq au$, then $aua^2=au$ and $a^2ua=ua$, so $a^2u=a^2ua^2=ua^2$. Thus $auau=au$ or $a^2u=ua^2$. If we replace $y$ with $z$ and repeat the same argument, then we get $a^2ua^2u=a^2u$ or $a^2u=ua^2$. Now we have the following 2 cases.\\
\begin{table}
      \small
      \centering
      \setlength\tabcolsep{2.5pt}
\begin{tabular}{c"c|c|c|c|c|c|c|c|c|c|c|c|c|c|c}
    $*$ & $a$ & $a^2$ & $v$ & $av$ & $va$ & $ava$ & $v^2$ & $av^2$ & $v^2a$ & $av^2a$\\\thickhline
    $a$ & $a^2$ & $a$ & $av$ & $a^2v$ & $ava$ & $va$ & $av^2$ & $v^2$ & $av^2a$ & $v^2a$\\\hline
    $a^2$ & $a$ & $a^2$ & $v$ & $av$ & $va$ & $ava$ & $v^2$ & $av^2$ & $v^2a$ & $av^2a$\\\hline
    $v$ & $va$ & $v$ & $v^2$ & $v$ & $v^2a$ & $va$ & $v$ & $v^2$ & $va$ & $v^2a$\\\hline
    $av$ & $ava$ & $av$ & $av^2$ & $av$ & $av^2a$ & $ava$ & $av$ & $av^2$ & $ava$ & $av^2a$\\\hline
    $va$ & $v$ & $va$ & $v$ & $v^2$ & $va$ & $v^2a$ & $v^2$ & $v$ & $v^2a$ & $va$\\\hline
    $ava$ & $av$ & $ava$ & $av$ & $av^2$ & $ava$ & $av^2a$ & $av^2$ & $av$ & $av^2a$ & $ava$\\\hline
    $v^2$ & $v^2a$ & $v^2$ & $v$ & $v^2$ & $va$ & $v^2a$ & $v^2$ & $v$ & $v^2a$ & $va$\\\hline
    $av^2$ & $av^2a$ & $av^2$ & $av$ & $av^2$ & $ava$ & $av^2a$ & $av^2$ & $av$ & $av^2a$ & $ava$\\\hline
    $v^2a$ & $v^2$ & $v^2a$ & $v^2$ & $v$ & $v^2a$ & $va$ & $v$ & $v^2$ & $va$ & $v^2a$\\\hline
    $av^2a$ & $av^2$ & $av^2a$ & $av^2$ & $av$ & $av^2a$ & $ava$ & $av$ & $av^2$ & $ava$ & $av^2a$
\end{tabular}
    \caption{$\left \langle a,v \right \rangle=\left \{ a,a^2,v,av,va,ava,v^2,av^2,v^2a,av^2a \right \}$ ($v=uau$)}
    \label{tab:15} 
\end{table} 
\begin{table}
    \begin{minipage}{.5\linewidth}
      \small
      \centering
      \setlength\tabcolsep{2.5pt}
\begin{tabular}{c"c|c|c|c|c|c|c|c}
    $*$ & $a$ & $a^2$ & $u$ & $u^2$ & $au$ & $a^2u$ & $au^2$ & $a^2u^2$ \\\thickhline
    $a$ & $a^2$ & $a$ & $au$ & $au^2$ & $a^2u$ & $au$ & $a^2u^2$ & $au^2$\\\hline
    $a^2$ & $a$ & $a^2$ & $a^2u$ & $a^2u^2$ & $au$ & $a^2u$ & $au^2$ & $a^2u^2$\\\hline
    $u$ & $au$ & $a^2u$ & $u^2$ & $u$ & $au^2$ & $a^2u^2$ & $au$ & $a^2u$\\\hline
    $u^2$ & $au^2$ & $a^2u^2$ & $u$ & $u^2$ & $au$ & $a^2u$ & $au^2$ & $a^2u^2$\\\hline
    $au$ & $a^2u$ & $au$ & $au^2$ & $au$ & $a^2u^2$ & $au^2$ & $a^2u$ & $au$\\\hline
    $a^2u$ & $au$ & $a^2u$ & $a^2u^2$ & $a^2u$ & $au^2$ & $a^2u^2$ & $au$ & $a^2u$\\\hline
    $au^2$ & $a^2u^2$ & $au^2$ & $au$ & $au^2$ & $a^2u$ & $au$ & $a^2u^2$ & $au^2$\\\hline
    $a^2u^2$ & $au^2$ & $a^2u^2$ & $a^2u$ & $a^2u^2$ & $au$ & $a^2u$ & $au^2$ & $a^2u^2$
\end{tabular}
      \caption{\small{$\left \langle a,u \right \rangle=\left \{ a,a^2,u,u^2,au,a^2u,au^2,a^2u^2 \right \}$}}
      \label{tab:16} 
    \end{minipage}%
    \begin{minipage}{.5\linewidth}
      \small 
      \centering
      \setlength\tabcolsep{2.5pt}
\begin{tabular}{c"c|c|c|c|c|c}
    $*$ & $a$ & $a^2$ & $w$ & $aw$ & $wa$ & $awa$ \\\thickhline
    $a$ & $a^2$ & $a$ & $aw$ & $w$ & $awa$ & $wa$ \\\hline
    $a^2$ & $a$ & $a^2$ & $w$ & $aw$ & $wa$ & $awa$ \\\hline
    $w$ & $wa$ & $w$ & $w$ & $w$ & $wa$ & $wa$ \\\hline
    $aw$ & $awa$ & $aw$ & $aw$ & $aw$ & $awa$ & $awa$ \\\hline
    $wa$ & $w$ & $wa$ & $w$ & $w$ & $wa$ & $wa$ \\\hline
    $awa$ & $aw$ & $awa$ & $aw$ & $aw$ & $awa$ & $awa$
\end{tabular}
    \caption{$\left \langle a,w \right \rangle$ ($w=aua$)}
    \label{tab:17} 
    \end{minipage}
\end{table} \\
(i) $a^2u=ua^2$ : Let $v$ be $uau$. Then $vav=uauauau=uau=v$, $a^2v=(a^2u)au=ua^2au=uau=v$ and $va^2=v$, so $\left \langle a,v \right \rangle=\left \{ a,a^2,v,av,va,ava,v^2,av^2,v^2a,av^2a \right \}$. Let $M_2$ be $\left \{ av,ava,av^2,av^2a \right \}$. Then from the table \ref{tab:15}, $M_2$, $M_2 \cup \left \{ a^2 \right \}$ and $M_2 \cup \left \{ a,a^2 \right \}$ are subsemigroups of $\left \langle a,v \right \rangle$. Thus $\left | \left \{ a,a^2,av,ava \right \} \right |\leq 3$. Suppose that $\left | \left \langle a,v \right \rangle \right |\geq 4$. If $av \in \left \{ a,a^2 \right \}$ or $ava \in \left \{ a,a^2 \right \}$, then $a(av)=v$ and $a(ava)a=v$ imply $v \in \left \{ a,a^2 \right \}$, a contradiction. Thus $av=ava$ and for the same reason, $ava=va$. Now $av=ava=va$ implies $\left \langle a,v \right \rangle=\left \{ a,a^2,v,av,v^2 \right \}$ and from the table \ref{tab:15}, $\left \{ a^2,v,av,v^2 \right \}$ is a subsemigroup of $\left \langle a,v \right \rangle$, a contradiction. Thus $\left | \left \langle a,v \right \rangle \right |\leq 3$, and this implies $\left | \left \{ a,a^2,v,av \right \} \right |\leq 3$. \\
If $a=av$ or $a^2=v$, then $a^2=(av)^2=av=a$, a contradiction. $a=v$ implies $a^2=av$. Thus $a^2=av$ or $v=av$. $a^2=av$ implies $a=aa^2=a(auau)=(ua^2)au=uau$. $v=av$ implies $aua=(auau)aua=(uau)aua=ua$ and for the same reason, $aua=au$. Thus $au=ua$ or $uau=a$. \\
If $au=ua$, then $\left \langle a,u \right \rangle=\left \{ a,a^2,u,u^2,au,a^2u,au^2,a^2u^2 \right \}$. Let $M_3$ be $\left \{ au,a^2u,au^2,a^2u^2 \right \}$. From the table \ref{tab:16}, it is easy to see that each of $M_3$, $M_3\cup \left \{ u^2 \right \}$, $M_3\cup \left \{ u,u^2 \right \}$ and $M_3\cup \left \{ a^2,u,u^2 \right \}$ are subsemigroups of $\left \langle a,u \right \rangle$. Thus $\left | \left \langle a,u \right \rangle \right |\leq 3$, a contradiction. \\
Now suppose that $uau=a$. Then $ya=au(uau)=(au^2)au=((uau)u^2)au=(uau)au=ay$ and $y^2=a(uau)=a^2$, so $\left \langle a,y \right \rangle=\left \{ a,a^2,y,ay \right \}$ and $\left | \left \langle a,y \right \rangle \right |\leq 3$. If $y \in \left \{ a,a^2 \right \}$ or $ay \in \left \{ a,a^2 \right \}$, then $\left \langle a,u \right \rangle=\left \{ a,a^2,u,u^2 \right \}$ and $\left | \left \langle a,u \right \rangle \right |\leq 3$, a contradiction. Thus $a^2u=au$ and for the same reason, $ua^2=ua$. This implies $au=ua$, which already proved to be impossible. \\
\\
(ii) $auau=au$ and $a^2ua^2u=a^2u$ : Let $w$ be $aua$. Then $w^2=a(a^2ua^2u)a=a(a^2u)a=w$, $waw=(auau)a=w$ and $a^2w=wa^2=w$ imply $\left \langle a,w \right \rangle=\left \{ a,a^2,w,aw,wa,awa \right \}$. Let $M_4$ be $\left \{ w,aw,wa,awa \right \}$. From the table \ref{tab:17}, it is easy to see that each of $M_4$, $M_4\cup \left \{ a^2 \right \}$ and $M_4\cup \left \{ a,a^2 \right \}$ are subsemigroups of $\left \langle a,w \right \rangle$. Thus $\left | \left \langle a,w \right \rangle \right |\leq 3$. If $aua \in \left \{ a,a^2 \right \}$ or $a^2ua \in \left \{ a,a^2 \right \}$, then $(aua)u=au$ and $a(aua)au=a^2u$ imply $a^2u=au$. If $a^2ua=aua$, then $a^2u=a(auau)=auau=au$. Thus $a^2u=au$ and for the same reason, $ua^2=ua$. \\ 
Now $\left \langle a,y \right \rangle=\left \{ a,a^2,y,ya \right \}$ implies $\left | \left \langle a,y \right \rangle \right |\leq 3$. $y^2=y$ and $(ay)^2=ay$ imply $y \notin \left \{ a,a^2 \right \}$, and $a(ya)=ya$ implies $ya \notin \left \{ a,a^2 \right \}$. Thus $ya=y$ and for the same reason, $aua=ua$ ($auau=au$ is equivalent to $uaua=ua$, and $a^2ua^2u=a^2u$ is equivalent to $ua^2ua^2=ua^2$). Thus $\left \langle a,u \right \rangle=\left \{ a,a^2,u,u^2,au \right \}$ and $\left \langle a^2,u \right \rangle=\left \{ a^2,u,u^2,au \right \}$, so $\left | \left \langle a,u \right \rangle \right |\leq 3$, a contradiction. 
\end{proof}

\vspace{3mm}

Thus we determine the complete answer to the Question. Table \ref{tab:18} shows the answer to the Question for each positive integer $r \geq 2$. 

\vspace{3mm}

\begin{table}[h]
\centering
\begin{tabular}{|c|c|}\hline
$r$ & $n$ \\\hline
2 & 1,2,4,6 \\\hline
$2^m+1$ ($m>0$) & 1,2,4 \\\hline
otherwise & 1 \\\hline
\end{tabular}
\caption{Results}
\label{tab:18} 
\end{table}

\vspace{3mm}

\section{References}

\begin{enumerate}

\item{Green, J.A., Rees, D.: On semigroups in which $x^r=x$. Proc. Camb. Phil. Soc. 48, 35-40 (1952)} \label{ref:1} 

\end{enumerate} 

\vspace{3mm}

Department of Mathematical Sciences, Seoul National University, Seoul 151-747, Korea 

e-mail: moleculesum@snu.ac.kr

\end{document}